\documentclass[11pt]{article}
\usepackage{amssymb,amsmath, amsthm, mathtools}
\usepackage{colordvi}
\usepackage{graphicx}
\usepackage{verbatim}
\usepackage{cancel}
\usepackage{caption}
\usepackage{epsf}
\usepackage{longtable}
\usepackage{multicol}
\usepackage{textcomp}
\usepackage[table]{xcolor}
\usepackage[includeheadfoot,left=2cm,right=2cm,top=2cm,bottom=2cm,head=15pt,
bindingoffset=0pt]{geometry}
\usepackage[all,  colour, line, dvips]{xy}
\usepackage{definitions}
\newtheorem*{pro}{Open Problem}
\def\pr{\begin{pro}}
\def\kpr{\end{pro}}
\def\horpath{\text{straight}}
\def\leftpath{\text{left}}
\def\rightpath{\text{right}}
\DeclareMathOperator{\Ker}{Ker}
\def\TRIPOD {
\begin{xy}
=(              0,  0) "0";
=(              0, 10) "A";
=(-\fiverootthree, -5) "B";
=( \fiverootthree, -5) "C";
<0mm,-0.1mm>;<0.2mm,-0.1mm>:;
"0" ; "A" **@{-};
"0" ; "B" **@{-};
"0" ; "C" **@{-};
\end{xy}}

\def\wmatrix#1{B_{\Omega#1}}
\newcommand\rootthree{1.73205}
\newcommand\fiverootthree{8.66025403784439}
\renewcommand\theenumi{(\roman{enumi})}

\author{Weronika Buczy\'nska \and Jaros\l{}aw Buczy\'nski \and  Kaie Kubjas \and Mateusz Micha\l{}ek}
\date{Sep 2nd, 2012}
\title{On the graph labellings arising from phylogenetics}
\begin{document}
\maketitle

\begin{abstract}
We study semigroups of labellings associated to a graph. These generalize the Jukes-Cantor model and phylogenetic toric varieties defined in \cite{buczynska_graphs}.
Our main theorem bounds the degree of the generators
of the semigroup by $g+1$ when the graph has first Betti number $g$.
Also, we provide a series of examples where the bound is sharp.
\end{abstract}

\medskip
{\footnotesize
\noindent\textbf{addresses:}\\
W.~Buczy\'nska, wkrych@mimuw.edu.pl,\\
\indent Institute of Mathematics of the Polish Academy of Sciences, \'Sniadeckich 8, 00-956 Warsaw, Poland\\
J.~Buczy\'nski, jabu@mimuw.edu.pl,\\
\indent Institute of Mathematics of the Polish Academy of Sciences, \'Sniadeckich 8, 00-956 Warsaw, Poland\\
K.~Kubjas, kaiekubjas@hotmail.com,\\
\indent  Institut f\"ur Mathematik, Freie Universit\"at Berlin, Arnimallee 3, 14195 Berlin, Germany\\
M.~Micha\l{}ek, wajcha2@poczta.onet.pl,\\
\indent Institute of Mathematics of the Polish Academy of Sciences, ul. \'{S}niadeckich 8, 00-956 Warszawa, Poland,\\
\indent and Max Planck Institute for Mathematics, Vivatsgasse 7, 53111 Bonn, Germany

\noindent\textbf{keywords:} 
graph labellings, phylogenetic semigroup, semigroup generators, lattice cone, Hilbert basis, conformal block algebras, Cavender-Farris-Neyman model,
2-state Jukes-Cantor model;

}

\section{Introduction}\label{sect_intro}

Throughout the article  $\ccG$ is a non-oriented graph.
We study a subset $\tau(\ccG)$ of the set of all labellings of edges of $\ccG$ by integers.
It has a natural structure of a graded semigroup with edge-wise addition (see \S\ref{sect_definitions} for the definition).
We call it the \textbf{phylogenetic semigroup of $\ccG$}, since the conditions on the labels come from phylogenetics. The first named author studied it in~\cite{buczynska_graphs} as a generalisation of the polytope defining
the Cavender-Farris-Neyman~\cite{Neyman}  model of a trivalent phylogenetic tree.
This model was studied in many papers and is often called the 2-state statistical Jukes-Cantor model
  \cite{buczynska_wisniewski}, \cite{sturmfels_sullivant}, \cite{pachter_sturmfels_alg_stat_for_comp_biology},
  and \cite{buczynska_graphs}\footnote{We thank Elizabeth Allman for bringing the original name of this model to our attention.}.
This is the simplest group-based model.
Hence the associated algebraic variety  is a toric variety, see~\cite{sturmfels_sullivant}, and it is the projective spectrum of $\CC[\tau(\ccG)]$.
%THINK: citation for toric variety.SUGGESTION: Historically first would be Handy, Penny or Evans Speed. However, I suggest giving sturmfels_sullivant or Weronika and Jarek's paper or my paper in J Alg.
Its equations are calculated in \cite{sturmfels_sullivant},
  and its geometric properties are examined in~\cite{buczynska_wisniewski}.

More recently Sturmfels and Xu \cite{sturmfels_xu}
proved that given the number of leaves $n$, the Jukes-Cantor model of a trivalent tree is a sagbi degeneration of the projective spectrum of the Cox ring of the blow-up of $\PP^{n-3}$ in $n$ points. This variety is closely related to the moduli space of rank $2$ quasi-parabolic vector bundles on $\PP^1$ with $n$ marked points.

Further work in this direction was done by Manon in \cite{manon_conformal_blocks} and \cite{manon_coordinate_rings}.
He used a sheaf of algebras over moduli spaces of genus $g$ curves with $n$ marked points coming from the conformal field theory.
The case $g=0$ is the construction of \cite{sturmfels_xu}, thus Manon's work generalises the Sturmfels-Xu construction.
The  semigroup algebras $\CC[\tau(\ccG)]$ are the toric deformations
  of the algebras over the most special points in the moduli of curves in Manon's construction.
Here $\ccG$ is the dual graph of the reducible curve represented by the special point.

Jeffrey and Weitsmann in~\cite{jeffrey_weitsman}
   studied the moduli space of flat $SU(2)$-connections on a genus $g$ Riemann surface.
In their context a trivalent graph $\ccG$ describes the geometry of the compact surface  of genus $g$ with $n$ marked points.
They considered a subset of $\ZZ$-labellings of the graph,
which is exactly $\tau(\ccG)_d$, the $d$-th graded piece of $\tau(\ccG)$.
They proved that the number of elements in this set is equal to the number of Bohr-Sommerfeld fibres associated
   to $\ccL^{\otimes d}$, where $\ccL$ is a natural polarising line bundle on the moduli space.
The Bohr-Sommerfeld fibres are also the central object of study in~\cite{jeffrey_weitsman}.
By the Verlinde formula \cite{verlinde_fusion_rules}, \cite{faltings_proof_for_Verlinde},
   the number of these fibres equals the dimension of the space of holomorphic sections of $\ccL^{\otimes d}$.
This number is the value of the Hilbert function of the toric model of a connected graph with first Betti number
   $g$ and with $n$ leaves (see \cite[Thm~8.3]{jeffrey_weitsman}
   and compare the conditions \cite[(8.2a--c)]{jeffrey_weitsman} with Lemma~\ref{def_cone_of_G} below).

Thanks to the Verlinde formula, which arises from mathematical physics,
  the Hilbert function of the semigroup algebra $\CC[\tau(\ccG)]$ has significant meaning.
In the case of trivalent trees it was used in~\cite{sturmfels_xu}
and studied by Sturmfels and Velasco in~\cite{sturmfels_velasco_P3_blow_up_and_spinor_vars}.
 One of the features of this model is that the Hilbert function depends only on the combinatorial data
\cite{buczynska_wisniewski}, \cite{buczynska_graphs}.
This phenomena fails to be true for other models, even group-based models
\cite{kubjas_kaie_hilb_poly_3Kimura_not_the_same},
\cite{michalek_donten_bury_Phylogenetic_invariants_for_group-based_models}.

As a summary, three distinct areas of science lead to study of the same object of purely combinatorial nature:
  the phylogenetic semigroup of a graph.
Firstly, it generalises Cavender-Farris-Neyman model of a phylogenetic tree.
Secondly, it is related to the moduli spaces of quasi-parabolic vector bundles and moduli spaces of marked curves.
Finally, conformal field theory is interested in enumerating elements of the semigroup.

In this paper we are interested in the problem of determining the degrees of elements in the minimal set of generators of the semigroup $\tau(\ccG)$.
Originally this problem was suggested to us by J.~Wi\'s{}niewski and B.~Sturmfels in a simplified version, where $\ccG$ is trivalent. Thanks to a suggestion of a referee we extended the results to arbitrary graphs.
First, we prove an upper bound for the degree of the generators in terms of first Betti number $g$ of the graph $\ccG$, see Section~\ref{sect_upper_bound} for the proof.

\begin{thm}\label{thm_upper_bound}
Let $\ccG$ be any graph with first Betti number $g$.
Any minimal generator of $\tau(\ccG)$ has degree at most $g+1$.
%The degree of each element in the minimal set of generators for $\tau(\ccG)$ is at most $g+1$.
\end{thm}

This result has been obtained in several special cases:
for trivalent trees i.e.~$g=0$,  in \cite{buczynska_wisniewski},
for arbitrary trees in \cite{michalek_donten_bury_Phylogenetic_invariants_for_group-based_models},
for trivalent graphs with $g=1$ in \cite{buczynska_graphs}.
%See Section~\ref{sect_upper_bound} for detailed reference.

Our second result shows that the upper bound of Theorem~\ref{thm_upper_bound} is attained for certain graphs.
We prove the theorem in Section~\ref{sect_construction_of_deg_g_plus_1_elts}
(see Example~\ref{claim_the_deg_g_plus_1_elt_is_indecomposable}).
See also Propositions~\ref{prop_omega_has_0} and Propositions~\ref{prop_G_has_two_valent}
   for extending the examples to graphs which are not trivalent, and with no loops.

\begin{thm}\label{thm_lower_bound_for_g_even}
Let $g$ be even.
There exists a graph $\ccG$ with first Betti number $g$
and an element $\omega\in \tau(\ccG)$ of degree $g+1$
which cannot be written as a non-trivial sum of two elements
$\omega = \omega' + \omega''$ for $\omega', \omega'' \in \tau(\ccG)$.
Specifically, $\ccG$ a $g$-caterpillar graph (see Figure~\ref{fig_g_caterpillar_graph}),
and $\omega$  the labelling in Figure~\ref{fig_the_indecomposable_elt_of_degree_g_plus_1}
is such an example.
\end{thm}

When $g$ is odd, for all trivalent graphs with first Betti number $g=1$ the bound is attained, as proved in \cite{buczynska_graphs}.
Also, there exist graphs with $g=3$, such that the bound is sharp.
The simplest of these is the $3$-caterpillar graph; we illustrate an indecomposable degree $4$ element
in Section~\ref{sect_small_examples}.
The odd case follows from the even case, i.e. Theorem~\ref{thm_lower_bound_for_g_even}.

\begin{cor}\label{cor_lower_bound_for_g_odd}
Let $g$ be odd.
There exists a graph $\ccG$ with first Betti number $g$,
and an element $\omega\in \tau(\ccG)$ of degree $g$
which cannot be written as a non-trivial sum of two elements in $\tau(\ccG)$.
Specifically, $\ccG$ can be taken as the $g$-caterpillar graph.
\end{cor}

It remains to address the case when $g$ is an odd integer greater than $3$.
It is natural to expect that when $\ccG$ is the $g$-caterpillar, then
there exists an indecomposable element in $\tau(\ccG)$ of degree $g+1$.
This, however, is false.

\begin{thm}\label{thm_even_decompose_on_caterpilar}
Suppose $\ccG$ is the $g$-caterpillar graph and $\omega \in \tau(\ccG)$ is an element of even degree $d\ge 6$.
   Then $\omega = \omega' +\omega''$ for some non-zero $\omega', \omega'' \in \tau(\ccG)$.
\end{thm}

In summary, the maximal degree of generators of the semigroup for the $g$-caterpillar graph is as follows.

\begin{cor}
Let $\ccG$ be a $g$-caterpillar graph. Then the semigroup $\tau(\ccG)$ is generated in degree
\[
 \begin{cases}
    g+1, & \text{if } g \text{ is even or } g\in\{1,3\} \\
      g, & \text{if } g \text{ is odd and } \ge 5.
 \end{cases}
\]
\end{cor}

Contrary to the case of the $g$-caterpillar graph, the conclusion of Theorem~\ref{thm_even_decompose_on_caterpilar} is false for some other graphs.
In Section~\ref{sect_small_examples} we present an indecomposable element of degree $6$ on a graph with first Betti number $6$.
However, we do not know if there exist a graph $\ccG$ with odd first Betti number $g\ge 5$
such that $\tau(\ccG)$ has a minimal generator of degree $g+1$.

A complete description of the generators of $\tau(\ccG)$ is known
for trivalent trees \cite{buczynska_wisniewski},
for trivalent graphs with first Betti number $1$ \cite{buczynska_graphs}.
We conclude by presenting results of some computational experiments.
Namely, we list all the generators of $\tau(\ccG)$ when $g\le 4$, and enumerate these generators when $g = 5$.

\subsection*{Acknowledgements}

K.~Kubjas was supported by DFG via the Berlin Mathematical School.
The remaining authors were supported by the research project
  ``Deformacje rozmaito\'sci algebraicznych ze specjaln\c a struktur\c a''
   funded by Polish Financial Means for Science in 2011.
W.~Buczy\'nska wishes to acknowledge the hospitality of  Institut Mittag-Leffler
   supported from the AXA Mittag-Leffler Fellowship Project and sponsored by the AXA Research Fund.
J.~Buczy\'nski thanks Institut Mittag-Leffler for hospitality and financial support during his visit at the Institute.
The authors would like to thank Christian Haase and Andreas Paffenholz for helpful discussions,
   Alexander Kasprzyk and Zach Teitler for their comments that have helped to improve the presentation.

\section{Semigroup associated with a graph}\label{sect_definitions}
In this section we generalize the construction of $\tau(\ccG)$ introduced for trivalent graphs in \cite{buczynska_graphs}.

\begin{defin}\label{defin_graph_path_network}
A \textbf{graph} $\ccG$ is a set $\ccV=\ccV(\ccG)$ of vertices and a set $\ccE=\ccE(\ccG)$ of edges,
which we identify with pairs of vertices.
We allow $\ccG$ to have loops or parallel edges.
A graph is \textbf{trivalent} if every vertex has valency one or three.
A vertex with valency one is called a \textbf{leaf} and an edge incident to a leaf is called a \textbf{leaf edge}.
A vertex that is not a leaf is called an \textbf{inner vertex}.
The set of inner vertices is denoted $\ccN=\ccN(\ccG)$.

A \textbf{path} is a sequence of pairwise distinct edges $e_0,\ldots,e_m$
with  $e_i \cap e_{i+1} \ne \emptyset$
for all $i \in \{0,\ldots, m-1\}$, such that either
both $e_0$ and $e_m$ contain a leaf, or $e_0 \cap e_{m} \ne \emptyset$.
In the latter case,
  the path is called a \textbf{cycle}.
A cycle of length one is a \textbf{loop}.
A graph with no cycles is a \textbf{tree}.
Two paths are \textbf{disjoint} if they have no common edge.
A \textbf{network} is a union of pairwise disjoint paths.
For consistency we say that the empty set is also a network.
An edge which is contained in a cycle is called \textbf{cycle edge}.
\textbf{First Betti number} of a graph is the minimal number of cuts that would make the graph into a tree.
\end{defin}

\begin{rmk}
Given the origins of the problem it is tempting to say  \emph{genus} of the graph instead of first Betti number.
However, this is inconsistent with the graph theory notation,
where  genus of a graph is the smallest genus of a surface such that the graph can be embedded into that surface.
\end{rmk}

\begin{defin}\label{lattice_defin}
Given a graph $\ccG$ let $\ZZ \ccE=\bigoplus_{e \in \ccE} \ZZ \cdot e$
be the lattice spanned by $\ccE$,
and $\ZZ \ccE^{\vee} = \Hom (\ZZ \ccE, \ZZ)$ be its dual.
Elements of the lattice $\ZZ \ccE$ are formal linear combinations of the edges,
thus $\ccE$ forms the standard basis of~$\ZZ \ccE$. The dual lattice~$\ZZ \ccE^{\vee}$
comes with the dual basis $\{ e^* \}_{e\in \ccE}$.
We define
\[
M = \{u\in \ZZ \ccE  \mid \forall v\in\ccN\ \quad \sum_{e \ni v}e^*(u)\in 2\ZZ\}.
\]
Then the graded lattice of the graph, with the degree map, is
\[
   \Mg=\ZZ \oplus M,
   \qquad
   \deg : \Mg = \ZZ \oplus M \ra \ZZ,
\]
given by the projection onto the first summand.
\end{defin}

% Then~\eqref{vertex_rep_in_N} becomes
% \[
% v=a_v+b_v+c_v.
% \]
%We first define phylogenetic semigroups on trees and then use the definition to define phylogenetic semigroups on graphs.

\begin{defin}
Given a tree $\ccT$ the \textbf{phylogenetic polytope $P(\ccT)$ on $\ccT$} is
\begin{displaymath}
 P(\ccT)=\textrm{conv}\{\sum_{e \in \ccE} a_ee\in M: a_e\in\{0,1\}\}.
\end{displaymath}
That is points in $P(\ccT) \cap M$ correspond to networks on $\ccT$.
The \textbf{phylogenetic semigroup $\tau(\ccT)$ on $\ccT$} is
\begin{displaymath}
 \tau(\ccT)=\textrm{cone}\{\{1\}\times P(\ccT)\}\cap \Mg.
\end{displaymath}
\end{defin}

The definition of the phylogenetic polytope on a tree corresponds to the definition of the polytope of the 2-state Jukes-Cantor binary model in \cite{michalek_group_based_models_are_pseudo_toric}, and in a different language in \cite{sturmfels_sullivant}. The phylogenetic semigroup on a tree is the semigroup associated to the phylogenetic polytope.

\begin{center}
\begin{minipage}{\textwidth}
\begin{center}
\epsfxsize=200pt\epsfbox{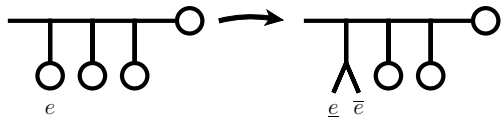}
\captionof{figure}{One step in the procedure of ``cutting'' a graph.}
\label{fig_cutting}
\end{center}
\end{minipage}
\end{center}

To a given graph $\ccG$ with first Betti number $g$ we associate a tree $\ccT$ with $g$ distinguished pairs of leaf edges.
This procedure can be described inductively on $g$.
If $g=0$, then the graph is a tree with no distinguished pairs of leaf edges.
For $g>0$ we choose a cycle edge $e$. We divide $e$ into two edges $\underline{e}$ and $\overline{e}$, adding two vertices $\underline{l}$ and $\overline{l}$ of valency~$1$. The edges $\underline{e}$ and $\overline{e}$ form a distinguished pair of leaf edges (see Figure~\ref{fig_cutting}).
This procedure decreases the first Betti number by one and increases the number of distinguished pairs by one.
Note that usually the resulting tree with distinguished pairs of leaf edges is not unique,
  however a tree with distinguished pairs of leaf edges encodes precisely one graph and the following definition does not depend on the resulting tree.

\begin{defin}
Let $\ccG$ be a graph.
Let $\ccT$ be the associated tree with a set of distinguished pairs of leaves $\{(\underline{e_i},\overline{e_i})\}$.
We define the \textbf{phylogenetic semigroup on $\ccG$} as
$$\tau(\ccG)=\tau(\ccT)\cap \bigcap_{i}\Ker (\underline{e_i}^*-\overline{e_i}^*).$$
In other words, the labelling on the $\underline{e_i}$ is identical to one on $\overline{e_i}$,
and thus the labelling descents to a labelling of $\ccG$.
Thus $\tau(\ccG)$ is canonically embedded in $\Mg(\ccG)$.
\end{defin}

We identify paths and networks in $\ccG$ as in Definition~\ref{defin_graph_path_network}
    with elements of the lattice $M$ and  replace
    union in $\ccE$ with sum in the group $M \subset \ZZ \ccE$.
Under this identification, the networks correspond precisely to the degree one elements in $\tau(\ccG)$.
More precisely, we define:

\begin{defin}
A \textbf{network in the graded lattice} $\Mg$ is a pair
  $\omega=(1,a) \in \Mg$ where $a\in M$ is a network.
\end{defin}

\section{The upper bound}\label{sect_upper_bound}

The goal of this section is to prove Theorem~\ref{thm_upper_bound}.
We proceed in three steps.
First we recall the result of \cite[Prop.~3.12]{michalek_donten_bury_Phylogenetic_invariants_for_group-based_models} that gives Theorem~\ref{thm_upper_bound} in case $g=0$:
if $\ccT$ is a tree, the phylogenetic polytope $P(\ccT)$ is \emph{normal},
meaning that  any lattice point in the rescaling $nP$ can be obtained as sum of $n$ lattice points in $P$ (usually not in a unique way).
This implies that the semigroup $\tau(\ccT)$ is generated by $\tau(\ccT)_1$.

\begin{cor}
   \label{cor_on_tree_elements_decompose}
   Let $\ccT$ be a tree. Every $\omega \in \tau(\ccT)_d$ can be expressed as
   \mbox{$\omega = \omega_1 +\dotsb + \omega_d$}, where each $\omega_i \in \tau(\ccT)_1$ is a network.
\end{cor}

In the second step, we represent a graph $\ccG$ with first Betti number $g$
   as a tree $\ccT$ together with $g$ distinguished pairs of leaf edges, that are ``glued'' together.
For an element $\omega \in \tau(\ccG)$ we consider the decomposition of the corresponding element in $\tau(\ccT)$
   into a sum of degree $1$ elements of $\tau(\ccT)$.
To each such decomposition we assign a matrix with entries in $\set{-1,0,1}$.
Since the decomposition is not unique, we study how simple modifications of the decomposition affect the matrix.
Finally, we apply a sequence of these modifications to the matrix
   to prove that any sufficiently high degree element of $\tau(\ccG)$ decomposes.

\subsection{Matrix associated to a decomposition of a lifted element}

Let $\ccG$ be a graph with first Betti number $g$ and $\ccT$ the associated tree with $g$ distinguished pairs of leaf edges. There is a one-to-one correspondence between elements of $\tau(\ccG)$ and the elements of $\tau(\ccT)$ that assign the same value to the leaf edges in each distinguished pair.
Thus we have the natural inclusion $\tau(\ccG)\subset \tau(\ccT)$.
See \cite[\S2.2--2.3]{buczynska_graphs} for a more geometric interpretation of this inclusion.

Let $\omega$ be an element of $\tau(\ccG)$.
By Corollary~\ref{cor_on_tree_elements_decompose}, in the semigroup $\tau(\ccT)$ there exists a decomposition
  $\omega=\omega_1+\dotsb+\omega_{\deg(\omega)}$, where $\omega_i\in \tau(\ccT)_1$.
Let $\Omega = (\omega_1,\dotsc,\omega_{\deg(\omega)})$
  and consider the matrix $\wmatrix{}$ with $\deg(\omega)$ rows and $g$ columns indexed by pairs of distinguished leaf edges.
The entry in the $i$-th row and column indexed by a pair of distinguished leaf edges $(\underline{e},\overline{e})$ is $\underline{e}^*(\omega_i)-\overline{e}^*(\omega_i)$.
Thus, since $\omega_i$ is a network $e^*(\omega_i) \in \set{0,1}$ for any edge,
   entries of $\wmatrix{}$ are only  $-1$, $0$ or $1$.

The matrix $\wmatrix{}$ depends on the tree $\ccT$ and on the decomposition of $\omega$ into the sum of degree one elements. An entry of $\wmatrix{}$ is zero when the corresponding network is compatible on the corresponding distinguished pair of leaf edges. Our aim is to decompose any element $\omega$ with $\deg(\omega)>g+1$ in $\tau(\ccG)$. This means that we are looking for decompositions in $\tau(\ccT)$ that are compatible on the distinguished pairs of leaf edges. Hence, it is natural to consider matrices with as many zero entries as possible.
\begin{lem}\label{prosty}
Let $\omega$ be an element of $\tau(\ccT)$. Let $\omega=\omega_1+\dots+\omega_{\deg(\omega)}$ be a decomposition of $\omega$ into networks. Let $\wmatrix{}$ be the matrix with $\deg(\omega)$ rows corresponding to the decomposition. For any subset of indices $\{j_1,\dots,j_p\}\subset \{1,\dots,\deg(\omega)\}$ the following conditions are equivalent:
\begin{enumerate}
\item the element $\omega_{j_1}+\dots+\omega_{j_p}$ is in $\tau(\ccG)$;
\item in each column of $\wmatrix{}$ the sum of entries in rows $j_1,\dots,j_p$ is equal to zero.
\end{enumerate} $\hfill\square$
\end{lem}
%\begin{proof}
%Let us consider a column indexed by a pair of distinguished leaf edges $(\underline{e},\overline{e})$.
%The sum of entries in this column from the rows $j_1,\dots,j_p$ is equal to zero if and only if
%\[
%\left( \underline{e}^*(\omega_{j_1})-\overline{e}^*(\omega_{j_1})\right) +\dots+ \left(\underline{e}^*(\omega_{j_p})-\overline{e}^*(\omega_{j_p})\right) =
%\underline{e}^*(\omega_{j_1}+\dots+\omega_{j_p})-\overline{e}^*(\omega_{j_1}+\dots+\omega_{j_p})=0,
%\]
%that is, if and only if the element $\omega_{j_1}+\dots+\omega_{j_p}$ is compatible for the pair $\underline{e}$ and $\overline{e}$.
%\end{proof}

Even if we start from a decomposable $\omega$ the associated matrix might not have this property; it depends upon the choice of decomposition of $\omega$ in $\tau(\ccT)$. The following lemma shows how to change this decomposition in order to obtain a matrix with the required property.

\begin{lem}\label{exchange_one_with_minus_one}
Let $\omega$ be an element of $\tau(\ccT)$. Let us choose a decomposition $\Omega$ of $\omega$ that gives a matrix $\wmatrix{}$ with as many zeros as possible.
Let us choose two entries in the matrix $\wmatrix{}$ that are in the same column indexed by $(\underline{e}_1,\overline{e}_1)$.
Suppose they are equal respectively, to $1$ and $-1$.
There exists a decomposition $\Omega'$ of $\omega$ that yields a matrix $\wmatrix{'}$ the same as $\wmatrix{}$,
 except for those two entries, which are interchanged.
\end{lem}
\begin{proof}
Let $\omega=\omega_1+\dots+\omega_{\deg(\omega)}$ be the given decomposition.
Without loss of generality we may assume that the entries are in the first and second row. Hence $\omega_1$ associates to the edges $\underline{e}_1$ and $\overline{e}_1$ values $0$ and $1$ respectively, and similarly $\omega_2$ associates $1$ and $0$.

To facilitate modifications of networks, we introduce the group of networks,
  following \cite[Def.~4.1]{michalek_toric_geometry_3Kimura}.
The elements are networks, and the group addition is modulo $2$,
  that is an edge is in the sum if and only if it is in exactly one of the summands. Formally:
\begin{notat}
The \textbf{group of networks} is the subset of
\[
\ZZ_2 \ccE:=\bigoplus_{e \in \ccE} \ZZ_2 \cdot e
\]
such that a formal sum $e_1+e_2 +\dotsb+e_k \in \ZZ_2 \ccE$ is in the group of networks
  if and only if $\set{e_1, e_2, \dotsc, e_k}$ is a network.
Note that this subset forms a subgroup of $\ZZ_2 \ccE$.

Let $S$ be the set of  all edges of the tree $\ccT$ on which the networks $\omega_1$ and $\omega_2$ disagree.
  $S$ is a network and in fact $S=\omega_1+\omega_2$ (sum in the group of networks).
Later we will replace $S$ with other networks.
\end{notat}

Our aim is to construct a network $b \subset S$ which realises the swapping of entries in the following sense.
For networks
   $\omega'_1 = \omega_1 + b$ and $\omega'_2 = \omega_2 + b$ (the sums in the group of networks),
   the new factorisation given by $\omega =  \omega'_1+\omega'_2+\omega_3+\dots+\omega_{\deg(\omega)}$
   (the sum is in $\tau(\ccT)$) interchanges the two entries as desired. The network $b$ will consist of paths
    $(p_1, p_2, p_3,\dotsc)$, which we construct inductively.
Define $p_1$ to be any path contained in $S$ starting at $\overline{e}_1$.
It is  possible as all inner vertices are adjacent to an even number of edges from  $S$.
  Next, we replace $S$ by $S+p_1$ (sum  in the group of networks).

Suppose that we have constructed a sequence of paths $p_1,\dots, p_{m-1}$ for $m>1$,
  where the first edge of $p_i$ is $\overline{e}_i$, the last is $\underline{e}_{i+1}$,
  and $(\underline{e}_i,\overline{e}_i)$ is a distinguished pair for all $i \in \setfromto{1}{m-1}$.
After each inductive step, if $p_m$ is constructed we will replace $S$ by $S+p_{m}$,
  where the sum is taken in the group of networks.

\begin{enumerate}
  \item\label{stop_not_paired}
        If the edge $\underline{e}_m$ is not paired, stop the construction.
        Otherwise go to Case~\ref{stop_pair_in_S}.
\item\label{stop_pair_in_S}
        If there is a distinguished pair $(\underline{e}_m,\overline{e}_m)$ and $\underline{e}_m^*(\omega_1)\neq \overline{e}_m^*(\omega_1)$ or $\underline{e}_m^*(\omega_2)\neq \overline{e}_m^*(\omega_2)$,
          i.e. at least one of the two entries in the column $(\underline{e}_m, \overline{e}_m)$ is non-zero,
          stop the construction.
        Otherwise go to Case~\ref{new_path_in_S}.
\item\label{new_path_in_S}
        If there is a distinguished pair $(\underline{e}_m,\overline{e}_m)$ and  $\underline{e}_m^*(\omega_1)=\overline{e}_m^*(\omega_1)$, $\underline{e}_m^*(\omega_2)=\overline{e}_m^*(\omega_2)$,
          then  $\omega_1$ and $\omega_2$ disagree on $\overline{e}_m$. Note that $\overline{e}_m$ is in $S$. Indeed, otherwise $\overline{e}_m$ would belong to some $p_i$ for $i<m$. As we have reached $\underline{e}_m$ by edges not belonging to any $p_i$ we must have $\underline{e}_m=\underline{e}_1$.
If this was true, we would have been in Case~\ref{stop_pair_in_S} and the construction would have terminated.

We define $p_{m}$ to be a path contained in $S$ starting from $\overline{e}_m$.
          Let $\underline{e}_{m+1}$ be the other end of the path $p_{m}$.
We increase $m$ by $1$ and replace $S$ by $S+p_{m}$, where the sum is taken in the group of networks.
We start over from Case~\ref{stop_not_paired}.
\end{enumerate}

Let us notice that the constructed paths are distinct, as each time we remove the edges of paths from $S$.
In particular, the construction terminates.
%THINK do we really want to delete the next 4 lines?
%Indeed, each path $p_{i+1}$ uniquely determines the path $p_i$.
%Hence the first path that would have been repeated is $p_1$.
%This is possible only if the previous path ends with $\underline{e}_1$.
%From the assumption, we would have been in Case~\ref{stop_pair_in_S}, hence the construction would terminate.

We define a network $b \subset S$ to be the union of paths $(p_1,\ldots,p_{m-1})$.
We use it to define two new networks $\omega_1'$ and $\omega_2'$.
Namely, $\omega_i'=\omega_i+b$, where the sum is taken in the group of networks.
In other words, $\omega_1'$ (resp.~$\omega_2'$) coincides with $\omega_1$ (resp.~$\omega_2$) on all edges apart from those belonging to the network $b$.
On the latter ones $\omega_1'$ (resp.~$\omega_2'$) is a negation of $\omega_1$ (resp.~$\omega_2$), hence coincides with $\omega_2$ (resp.~$\omega_1$).
In particular,  $\omega_1+\omega_2=\omega_1'+\omega_2'$, where this time the sum is taken in $\tau(\ccT)$.

We get a decomposition $\Omega'=(\omega_1', \omega_2', \omega_3, \dotsc, \omega_{\deg(\omega)})$
  with $\omega = \sum \Omega'$ and the associated matrix $\wmatrix{'}$.
We claim that it exchanges the two chosen entries equal to $1$ and $-1$.

Consider each distinguished pair of leaf edges through which we passed during our construction of $(p_1,\ldots,p_{m-1})$.
If we did not stop at a pair $(l_1,l_2)$ each network $\omega_1$ and $\omega_2$ assigns the same value to $l_1$ and $l_2$
  ---  otherwise we would have stopped because of Case~\ref{stop_pair_in_S}.
On these leaf edges $\omega_1'$ and $\omega_2'$ agree  with $\omega_2$ and $\omega_1$ respectively.
Hence, they also assign the same value to $l_1$ and $l_2$.
In particular, both  $\wmatrix{}$ and $\wmatrix{'}$ have zeros in the first two rows in the column indexed by $(l_1,l_2)$.
In fact, the only four entries on which $\wmatrix{}$ and $\wmatrix{'}$ might possibly differ are the entries
  in first two rows in the columns indexed by $(\underline{e}_1,\overline{e}_1)$ or $(\underline{e}_m,\overline{e}_m)$, where $p_m$ is the last path.

Suppose the construction stopped in \ref{stop_not_paired}.
Then the last leaf edge is not paired, hence we change only entries in the column indexed by $(\underline{e}_1,\overline{e}_1)$.
Since both $\omega_1'$ and $\omega_2'$ agree on $\underline{e}_1$ and $\overline{e}_1$,  we have that  $\wmatrix{'}$ has two zeros,  whereas $\wmatrix{}$ had $1$ and $-1$.
This contradicts the assumption that $\wmatrix{}$ has as many zeroes as possible.

Now suppose the construction terminated in Case~\ref{stop_pair_in_S}.
Consider two sub-cases.

1)
The edges $\underline{e}_m\neq \underline{e}_1$ are distinct.
We  exclude this case.
We change four entries in two columns.
The two entries in the column indexed by $(\underline{e}_1,\overline{e}_1)$ are changed from $1$ and $-1$ to zero.
We know that matrix $\wmatrix{'}$ has at most as many zero entries as $\wmatrix{}$.
Hence the two entries in the column indexed by $(\underline{e}_m,\overline{e}_m)$ must be changed from two zeros to two non-zeros.
Having two zeros in $\wmatrix{}$ in those entries contradicts the assumptions of Case~\ref{stop_pair_in_S}.

2) The edges $\underline{e}_m=\underline{e}_1$ are equal. In this case $\overline{e}_m=\overline{e}_1$, so we only exchange two entries in the column indexed by $(\underline{e}_1,\overline{e}_1)$. This means that we have exchanged $1$ and $-1$, which proves the lemma.
\end{proof}

\begin{proof}[Proof of Theorem~\ref{thm_upper_bound}]
Consider an element $\omega$ of degree $\deg(\omega)>g+1$ in  $\tau(\ccG)$ and
a tree $\ccT$ associated with the graph $\ccG$.
Let us choose a decomposition $\Omega$ of $\omega$ in $\tau(\ccT)$,
  so that the associated matrix $\wmatrix{}$ has as many zero entries as possible.
We find a subset of rows of the matrix $\wmatrix{}$ such that the sum of entries in each column is even as follows.
Reduce the entries of $\wmatrix{}$ modulo 2 obtaining the matrix $C_{\Omega}$ with entries from $\ZZ_2$.
We think of rows of $C_{\omega}$ as vectors of the $g$-dimensional vector space over the field $\ZZ_2$.
We have $\deg(\omega)>g+1$ such vectors. Hence we can find a \emph{strict} subset of linearly dependent vectors.
As we work over $\ZZ_2$ there exists a strict subset of these vectors summing  to $0$. The same subset $R$ of rows in matrix $\wmatrix{}$ sums to even numbers in each column.

Since $\omega \in \tau(\ccG)$, the sum of entries in each column of the matrix $\wmatrix{}$ is zero. Suppose the sum of entries in the rows from $R$ is non-zero in a column. Using Lemma~\ref{exchange_one_with_minus_one} we exchange the entries, changing the sum by $2$ until it is equal to zero. This way we get a decomposition $\Omega'$ of $\omega$ such that the rows from $R$ sum to zero in each column. By Lemma~\ref{prosty} the sum of networks corresponding to rows from $R$ is in $\tau(\ccG)$.
The sum of the remaining networks is also in $\tau(\ccG)$.
We have obtained a non-trivial decomposition of~$\omega$.
\end{proof}

%THINK: figure depicting a factorisation as above.

\section{The upper bound is sharp for even $g$}\label{sect_construction_of_deg_g_plus_1_elts}

In this section we show that if $g$ is even, the bound $g+1$ is sharp for a caterpillar graph with $g$ loops.
More generally to construct high degree indecomposable elements it suffices  to consider  trivalent graphs.

\begin{lem}
 Let $\ccG$ be a graph with first Betti number $g$ and phylogenetic semigroup  generated in degree $n$. There exists a trivalent graph $\ccG'$ with first Betti number $g$ and phylogenetic semigroup  generated in degree $\geq n$.
\end{lem}

\begin{proof}
 We construct $\ccG'$ from $\ccG$. Choose an inner vertex $v$ of $\ccG$ that is not trivalent. Replace $v$ by $v'$ and $v''$ together with a new edge between them, let 2 edges incident to $v$ be incident to $v'$ and the rest of the edges incident to $v$ be incident to $v''$. After a finite number of replacements we get a trivalent graph $\ccG'$,
   because $valency(v')< valency (v)$ and $valency(v'')< valency (v)$.

Now consider a tree $\ccT$ with $g$ distinguished pairs of leaf edges associated to $\ccG$ that is attained by dividing edges $e_1,e_2,\ldots,e_g$ into two. Dividing exactly the same edges $e_1,e_2,\ldots,e_g$ into two in $\ccG'$ gives a tree $\ccT'$ with $g$ distinguished pairs of leaf edges associated to $\ccG'$. As $\tau(\ccT)$ and $\tau(\ccT')$ are normal, the semigroup $\tau(\ccT)$ is a coordinate projection of the semigroup $\tau(\ccT')$ that forgets coordinates corresponding to new edges.
Hence the semigroup $\tau(\ccG)$ is a coordinate projection of the semigroup $\tau(\ccG')$ and projections of generators of $\tau(\ccG')$ generate $\tau(\ccG)$.
\end{proof}

\subsection{Trivalent graphs}

We introduce  notation and definitions specific to trivalent graphs useful for constructing high degree indecomposable elements.

\begin{notat}
We denote the elements of the lattice
$\ZZ \ccE^{\vee}$ dual to the edges meeting at the inner vertex~$v$
\[
\begin{array}{lcr}
a_v:=\big(i_v(e_1)\big)^*,  &
b_v:=\big(i_v(e_2)\big)^*,  &
c_v:=\big(i_v(e_3)\big)^*,  \\
\end{array}
\]
where $\{e_1, e_2, e_3\}$ are the edges of $\TRIPOD$ and
$i_v : \TRIPOD \mono \ccG$ is a map which is
locally an embedding and sends the central vertex of $\TRIPOD$
to $v$ (see Figures~\ref{fig_example_for_abc} and \ref{fig_decomposition_into_local_paths}).
\end{notat}

Given an element $\omega$ in either $\ZZ \ccE$, $M$, or $\Mg$,
each of $a_v, b_v, c_v \in \ZZ \ccE^{\vee}$ measures the coefficient of $\omega$
at an edge incident to $v$.

\begin{defin}
The \textbf{degree} of $\omega\in \Mg$ \textbf{at an inner vertex} $v\in\ccN$ is
\[
\deg_v(\omega):=\half\cdot\bigl(a_v(\omega)+b_v(\omega)+c_v(\omega)\bigr)
\]
(see Figure~\ref{fig_example_for_abc2} for an illustrative example).
\end{defin}

\begin{figure}[htb]
\begin{tabular}{ccc}
\begin{minipage}[t]{0.43\textwidth}
\begin{center}
\epsfxsize=150pt\epsfbox{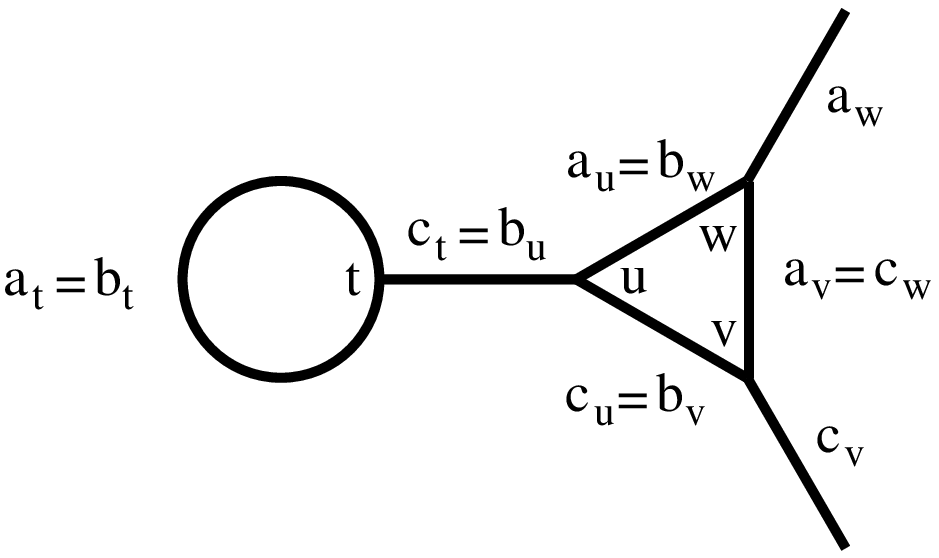}
\end{center}
\caption{A graph with four vertices $t, u, v, w$ with $a_*, b_*, c_*$ indicated for each vertex.}
\label{fig_example_for_abc}
\end{minipage}
 &   \phantom{b} &
\begin{minipage}[t]{0.43\textwidth}
\begin{center}
\epsfxsize=150pt\epsfbox{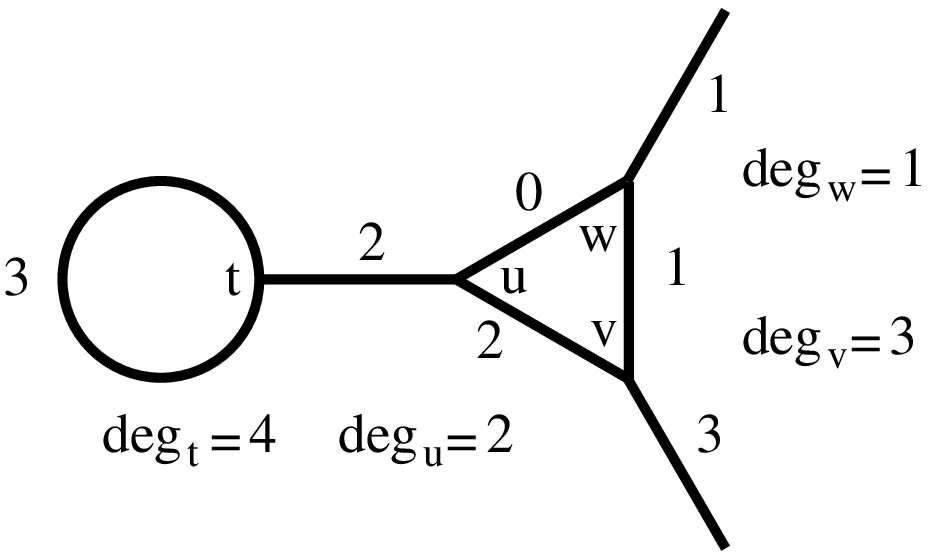}
\end{center}
\caption{An element $\omega \in \tau(\ccG)$ with $\deg(\omega)=4$, together with values of degrees at each inner vertex.
         Thus $a_t(\omega) = b_t(\omega) = 3$, $c_t(\omega)=  b_u(\omega)= 2$, etc.}
\label{fig_example_for_abc2}
\end{minipage}
\end{tabular}
\end{figure}

Following \cite[Def.~2.18 \& Lem.~2.23]{buczynska_graphs}
  we give the inequality description of phylogenetic semigroups for trivalent graphs.

\begin{lem}\label{def_cone_of_G}
For a trivalent graph $\ccG$ the \textbf{phylogenetic semigroup $\tau(\ccG)$ on $\ccG$} is the set of elements $\omega$ satisfying the following conditions
\begin{enumerate}
\renewcommand{\theenumi}{\textnormal{[$\heartsuit\!\!\heartsuit$]}}
\item  \label{item_parity_condition} parity condition:  $\omega \in \Mg$,
\renewcommand{\theenumi}{\textnormal{[+]}}
\item  \label{item_non-negativity_condition} non-negativity condition:
$e^*(\omega) \ge 0$ for any $e \in \ccE$,
\renewcommand{\theenumi}{\textnormal{[$\triangle$]}}
 \item  \label{item_triangle_inequalities} triangle inequalities:
$ |a_v(\omega)-b_v(\omega)| \leq c_v(\omega) \leq a_v(\omega)+ b_v(\omega)$, for each inner vertex $v\in \ccN$,
\renewcommand{\theenumi}{\textnormal{[\textdegree]}}
\item  \label{item_degree_inequality}
degree inequalities:  $\deg(\omega) \geq \deg_v(\omega)$ for any $v\in \ccN$.
\end{enumerate}
\end{lem}

The triangle inequalities~\ref{item_triangle_inequalities} are symmetric and do not depend on the embedding $i_v$.
\begin{rmk}\label{value_less_than_degree}
   If every edge of $\ccG$ contains at least one inner vertex,
      then the inequalities above imply $\deg(\omega) \geq e^*(\omega)$
      for all edges.
   On the other  hand, in the degenerate cases
      where one of the connected components of $\ccG$ consists of one edge only,
      for consistency the inequality $\deg(\omega) \geq e^*(\omega)$ should be included in Lemma~\ref{def_cone_of_G}.
   However, we will not consider these degenerate cases here.
\end{rmk}

\subsection{Loops, caterpillar graphs, and local paths}

Assume $\ccG$ is trivalent. We investigate the influence of loops in the graph $\ccG$ on the semigroup $\tau(\ccG)$, particularly on the parity condition. Then we define the $g$-caterpillar graph and apply the conditions coming from loops to this case.
Finally, we define an element of the phylogenetic semigroup and we prove it is indecomposable.

\begin{ex}
  Let $o \in \ccE$ be a loop with unique vertex $v_o \in o$.
There is exactly one edge $e_o$, other than $o$, such that $v_o \in e_o$.
Loops force parity of the label on $e_o$,
   that is if $\omega \in M$ then $e_o^*(\omega)$ is even.
It is a straightforward consequence of the parity condition \ref{item_parity_condition}, or the definition of $M$
  in the neighbourhood of $v_o$.
\end{ex}

\begin{center}
\begin{minipage}{\textwidth}
\begin{center}
\epsfxsize=200pt\epsfbox{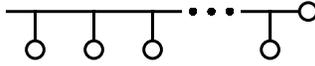}
\captionof{figure}{The $g$-caterpillar graph.}
\label{fig_g_caterpillar_graph}
\end{center}
\end{minipage}
\end{center}
The trivalent graph obtained from the caterpillar tree with $g+1$ leaves by attaching a loop to all but one leaf
(the leftmost one) is called the \textbf{$g$-caterpillar graph}, see Figure~\ref{fig_g_caterpillar_graph}.

\begin{ex}\label{lem_parity_on_caterpillar}
Let $\ccG$ be the $g$-caterpillar graph and $\omega \in \ZZ\ccE$.
The parity condition~\ref{item_parity_condition} on the $g$-caterpillar graph can be seen as a requirement of
  parity at each edge which is not a loop. That is
$\omega \in M$ if and only if
$e^*(\omega)$ is even on every edge $e$ other than loops.
\end{ex}

The conditions defining $\tau(\ccG)$
   imply that every element $\omega\in\tau(\ccG)$ decomposes locally in a unique way into paths around any vertex.
This means that there exist non-negative integers $x_v, y_v, z_v$ related to $a_v, b_v, c_v$ as in Figure~\ref{fig_decomposition_into_local_paths} such that $\deg(\omega) \ge x_v +y_v + z_v$
   (see \cite[Sect.~2.4]{buczynska_graphs} for more details).
In the case of the $g$-caterpillar graph we denote the local paths at an inner vertex $v$ on the horizontal line \emph{straight} ($z_v$), \emph{left} ($y_v$) and \emph{right} ($x_v$) paths,
  see Figure~\ref{fig_building_graph}.
A consequence of Example~\ref{lem_parity_on_caterpillar} in terms of the local paths is the following.
\begin{center}
\begin{tabular}{ccc}
\begin{minipage}[b]{0.49\textwidth}
\[
\begin{xy}
(0,0)*{}="0";
(0,10)*{} ="A";
(-2,0)="Abvector";
(1,-\rootthree)="Bcvector";
(1, \rootthree)="Cavector";
( 1,10)*{} ="Ac";
(-\fiverootthree,-5)*{}="B" ;
( \fiverootthree,-5)*{}="C" ;
(0,-2)="a"; ( \halfrootthree,0.5)="b"; (-1.5,  \halfrootthree)="c";
"0" ; "A" **@{-}-"a"*{a_v};
"0" ; "B" **@{-}-"b"*{b_v};
"0" ; "C" **@{-}-"c"*{c_v};
"A"+"Abvector" ; "B"-"Bcvector" **\crv{"0"+"c"};
"B"+"Bcvector" ; "C"-"Cavector" **\crv{"0"+"a"};
"C"+"Cavector" ; "A"-"Abvector" **\crv{"0"+"b"};
"0"-"A" *{x_v};
"0"-"B" *{y_v};
"0"-"C" *{z_v};
\end{xy}
\qquad
\begin{array}{rlrlrlrl}
a_v&=&     & & y_v &+& z_v\\
b_v&=& x_v &+&     &+& z_v\\
c_v&=& x_v &+& y_v & &    \\
\end{array}
\]
\captionof{figure}{The decomposition into local paths at any vertex.}
\label{fig_decomposition_into_local_paths}
\end{minipage}&\phantom{ab}&
\begin{minipage}[b]{0.4\textwidth}
\begin{center}
\epsfxsize=80pt\epsfbox{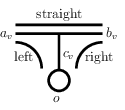}
\captionof{figure}{Notation for local paths on a vertex in a $g$-caterpillar graph.}
\label{fig_building_graph}
\end{center}
 \end{minipage}
\end{tabular}
\end{center}

\begin{cor}\label{cor_parity_on_caterpillar_in_terms_of_local_paths}
  Let $\ccG$ be a $g$-caterpillar graph,  $\omega \in \tau(\ccG)$,  $v$ a vertex not on a loop. Then
    \begin{itemize}
     \item  if $\deg_v \omega$ is even,  $x_v(\omega), y_v(\omega), z_v(\omega)$ are all even,
     \item  if $\deg_v \omega$ is odd,  $x_v(\omega), y_v(\omega), z_v(\omega)$ are all odd.
    \end{itemize}
\nopagebreak
  In particular, $\deg_v(\omega) \ne 1$.%\noprf
\end{cor}

%\subsection{An indecomposable element}
\begin{ex}\label{claim_the_deg_g_plus_1_elt_is_indecomposable}
Suppose $g=2k$ is even, and let $\ccG$ be the $g$-caterpillar graph.
The element $\omega$ defined on Figure \ref{fig_the_indecomposable_elt_of_degree_g_plus_1} is indecomposable.

\begin{center}
\epsfxsize=200pt\epsfbox{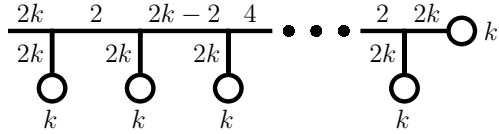}
\captionof{figure}{The indecomposable element $\omega$ of degree $g+1$ on the $g$-caterpillar graph for even~$g$.}
\label{fig_the_indecomposable_elt_of_degree_g_plus_1}
\end{center}

\end{ex}

\begin{proof}
We begin the proof by explaining the local decomposition of $\omega$.
Starting from the left-most inner vertex of the caterpillar tree we have

\noindent
\begin{longtable*}{rll}
(1) & $2k-1$ left, $1$ right, $1$ straight paths & \parbox[c]{151pt}{\epsfxsize=150pt\epsfbox{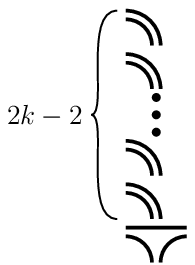}}\\
\\
(2) & $2k-2$ right, $2$ left paths &               \parbox[c]{151pt}{\epsfxsize=150pt\epsfbox{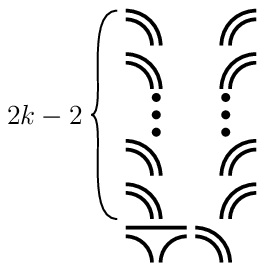}}\\
\\
(3) & $2k-3$ left, $3$ right, $1$ straight paths & \parbox[c]{151pt}{\epsfxsize=150pt\epsfbox{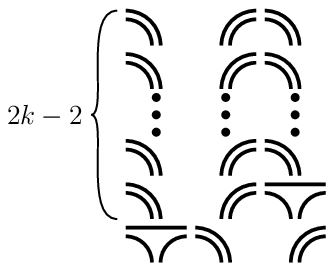}}\\
$\vdots$ & $\vdots$ \\
($2k-1$) &  $1$ left, $2k-1$ right, $1$ straight paths & \parbox[c]{151pt}{\epsfxsize=150pt\epsfbox{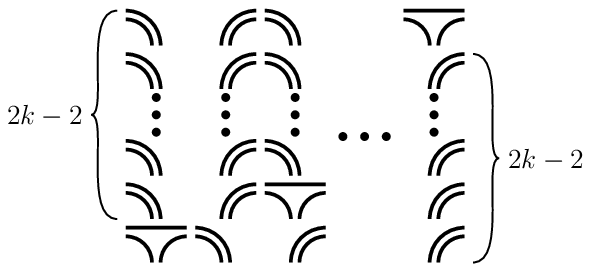}}
\end{longtable*}

\medskip

Suppose for a contradiction that $\omega$ is decomposable as $\omega'+\omega''$.
Since the degree of $\omega$ is odd, one of the two parts has even degree. Assume $\omega'$ has even degree $\deg(\omega')=2i$ with $i >0$.

Every second vertex $v$ on the horizontal line has a single straight line in the local decomposition of $\omega$.
Moreover at such $v$ the degree is attained $\deg_v \omega = \deg \omega$.
Thus  $\deg_v \omega' = \deg \omega'$ and $\deg_v \omega'' = \deg \omega''$ as well.
By Corollary~\ref{cor_parity_on_caterpillar_in_terms_of_local_paths} the local decomposition of $\omega''$ at $v$
  consists of the single straight path and odd number of left paths and odd number of right paths,
whereas the local decomposition of $\omega'$ at $v$
  consists of even number of left paths and even number of right paths.
This means $\omega'$ must have $2i$ left paths at the left-most inner vertex on the horizontal line of $\ccG$.
At the next inner vertex on the horizontal line, $\omega'$ has $2i$ right paths by Example~\ref{lem_parity_on_caterpillar}, and so on.
This is a contradiction, as at some inner vertex on the horizontal line $\omega$ has less than $2i$ left paths.
\end{proof}

\section{A lower bound for odd $g$}\label{sect_bound_not_sharp_on_odd_g}

If $g$ is odd, there exist graphs with first Betti number $g$
   with minimal generators of $\tau(\ccG)$ in degree $g$.
They are obtained by extending the labelling from Example~\ref{claim_the_deg_g_plus_1_elt_is_indecomposable}
    to the extra loop of the $(g+1)$-caterpillar.
We do not know if the maximal generating degree is $g$ or $g+1$
    among the graphs with first Betti number $g$.
However, we know it for the $g$-caterpillar graph.

\begin{lem}\label{lem_caterpillar_even_degree}
Let  $\ccG$ be the $g$-caterpillar graph.
Let $\omega\in\tau(\ccG)$ be an element of even degree at least $6$. Then $\omega$ can be decomposed into degree $2$ and $\deg(\omega)-2$ elements.
\end{lem}

\begin{proof}
For this proof we fix the following notation. At each vertex $v$ on the horizontal line of the $g$-caterpillar, we choose an embedding of the tripod so that $a_v$, $b_v$ and $c_v$ are arranged as in Figure~\ref{fig_building_graph}, so $c_v$ is the value on the vertical edge, $a_v$ on the left one, $b_v$ on the right one.

Let  $d:=\deg(\omega)$ be the degree of $\omega$. We will define a degree $2$ element $\omega'$, so that $\omega=\omega'+\omega''$ is a decomposition in $\tau(\ccG)$. In our construction we use local paths. This assures that the resulting $\omega'$ and $\omega''$  fulfill the triangle inequalities \ref{item_triangle_inequalities} of $\tau(\ccG)$. To assure the degree inequalities \ref{item_degree_inequality}, we require that $\omega'$ satisfies the following  at each inner vertex~$v$
\begin{align}\label{condition_for_omega_prime}
d-2\ge \text{deg}_v(\omega'')=\text{deg}_v(\omega)-\text{deg}_v(\omega').
\end{align}
Note that if $\text{deg}_v(\omega')=2$, or equivalently, if $\omega'$ is constructed using two local paths at $v$, then (\ref{condition_for_omega_prime}) is automatically fulfilled.

First we define the labels of $\omega'$ on the caterpillar tree, ignoring the labels on the loops for a while.
We define them inductively from left to right using local paths,
   in such a way that the following condition holds for every inner vertex $v$ of the caterpillar tree

\begin{align}\label{inductive_condition_for_omega}
 b_v(\omega')=\left\{
  \begin{array}{l l}
0 & \text{if } b_v(\omega)<   \frac{d}{2},\\
2 & \text{if } b_v(\omega)>   \frac{d}{2},\\
0 \text{ or } 2 & \text{otherwise}.
\end{array} \right.
\end{align}
First we define $\omega'$ for the left-most edge $e$
\begin{align*}
 e^*(\omega')=\left\{
  \begin{array}{l l}
0 & \textrm{if } e^*(\omega)\leq \frac{d}{2},\\
2 & \textrm{otherwise}.
\end{array} \right.
\end{align*}
We need to prove that at every step there is enough of local paths in $\omega$
to fulfill conditions~(\ref{condition_for_omega_prime}) and~\eqref{inductive_condition_for_omega}.
There are six cases depending on the value of $\omega'$ on the previous edge and the value of $\omega$ on the current one.

\begin{enumerate}
\item If $a_v(\omega')=2$ and $b_v(\omega)>d/2$, then we have to prove that $\omega$ has at least two straight paths at $v$, since we need $b_v(\omega')=2$. The condition~\eqref{inductive_condition_for_omega} gives $a_v(\omega)\ge d/2$, and
\[
\# \horpath= z_v(\omega)=\frac{a_v(\omega)+b_v(\omega)-c_v(\omega)}{2} > \frac{d}{2} - \frac{c_v(\omega)}{2} >0,
\]
where the last inequality holds because of
\[
 2d\ge a_v(\omega)+b_v(\omega)+c_v(\omega)>\frac{d}{2}+\frac{d}{2}+c_v(\omega)=d+c_v(\omega).
\]

As $d$ and $c_v(\omega)$ are both even, we conclude that $\omega$ has at least two straight paths at $v$.

\item\label{item_condition_a=2_b=d/2}
 If $a_v(\omega')=2$ and $b_v(\omega)=d/2$, then we have to prove that $\omega$ has either at least two straight paths or at least two left paths at $v$, since we need either $b_v(\omega')=2$ or $b_v(\omega')=0.$ The condition~\eqref{inductive_condition_for_omega} gives $a_v(\omega)\ge d/2$, and
\[
\# \horpath +\# \leftpath = z_v(\omega) + y_v(\omega) = a_v(\omega)\ge \frac{d}{2} \ge 3.
\]

\item If $a_v(\omega')=2$ and $b_v(\omega)<d/2$, then we have to prove that $\omega$ has at least two left paths at $v$, since we need $b_v(\omega')=0$. The condition~\eqref{inductive_condition_for_omega} gives $a_v(\omega)\ge d/2$, and thus $a_v(\omega)-b_v(\omega)>0$.
By the triangle inequalities \ref{item_triangle_inequalities} we have
\begin{multline*}
%\[
\# \leftpath=y_v(\omega)=\frac{c_v(\omega)+a_v(\omega)-b_v(\omega)}{2} \ge
\frac{a_v(\omega)-b_v(\omega)}{2} +
\frac{|a_v(\omega)-b_v(\omega)|}{2} =
a_v(\omega) - b_v(\omega) \ge 1
%\]
\end{multline*}
As  $a_v(\omega)$ and $b_v(\omega)$ are both even, we conclude that $\omega$ has at least two left paths at $v$.

\item If $a_v(\omega')=0$ and $b_v(\omega)>d/2$, then we have to prove that $\omega$ has at least two right paths at $v$, since we need $b_v(\omega')=2$. The condition~\eqref{inductive_condition_for_omega} gives $a_v(\omega)\le d/2$, and thus $b_v(\omega)-a_v(\omega)>0$.
Again, by the triangle inequalities \ref{item_triangle_inequalities} we have
\begin{multline*}
\# \rightpath=x_v(\omega)=\frac{c_v(\omega)+b_v(\omega)-a_v(\omega)}{2} \ge
\frac{ b_v(\omega)-a_v(\omega)}{2} +
\frac{|b_v(\omega)-a_v(\omega)|}{2} = b_v(\omega) - a_v(\omega) \ge 1
\end{multline*}

\item\label{item_condition_a=0_b=d/2}
 If $a_v(\omega')=0$ and $b_v(\omega)=d/2$, we have to prove that either $\deg_v(\omega)\le d-2$ or $\omega$ has two right paths at $v$, since we need either $b_v(\omega')=0$ or $b_v(\omega')=2.$
If $\deg_v(\omega)\ge d-1$, using the condition \eqref{inductive_condition_for_omega} gives
\[
 \# \rightpath=x_v(\omega)= \frac{b_v(\omega)+c_v(\omega)-a_v(\omega)}{2} = \deg_v(\omega)-a_v(\omega) \ge (d-1)-\frac{d}{2} \ge 2.
\]
%We conclude that either $\deg_v(\omega)\le d-2$ or $\omega$ has at least two right paths at~$v$.

\item If $a_v(\omega')=0$ and $b_v(\omega)<d/2$, then we have to prove that $\deg_v(\omega)\le d-2$, since we need $b_v(\omega')=0.$ The condition \eqref{inductive_condition_for_omega} gives $a_v(\omega)\le d/2$, and thus $a_v(\omega)+b_v(\omega) \le d-1$. As $a_v(\omega)$ and $b_v(\omega)$ are both even, we even have
         $a_v(\omega)+b_v(\omega) \le d-2$.
       Using this and the triangle inequalities \ref{item_triangle_inequalities}, we get the desired inequality:
       \[
        2\deg_v(\omega)=a_v(\omega)+b_v(\omega)+c_v(\omega)\le d-2+c_v(\omega)\le d -2 + a_v(\omega)+b_v(\omega) \le 2d - 4.
       \]
\end{enumerate}
Note that we use $d\ge 6$ only in cases with $b=d/2$,
  i.e., cases \ref{item_condition_a=2_b=d/2} and \ref{item_condition_a=0_b=d/2}).

It remains to suitably define the labels of $\omega'$ on the loops.
Fix a loop $o$.
In the local decomposition of $\omega$ at the vertex $v_o$ some of the local paths come in pairs:
  There are $e^*_o(\omega)/2$ loops with $2$ on the adjacent edge and $1$ on the loop;
  there are $(o^*(\omega)-e^*_o(\omega)/2)$ single loops with $0$ on the adjacent edge and $1$ on the loop.

If $e^*_o(\omega')=2$
   then $e^*_o(\omega) \ge 2$, and there is at least one loop with $2$ on the adjacent edge in
   the local decomposition of $\omega$.
Set $o^*(\omega')=1$.

Otherwise $e^*_o(\omega')=0$ by the construction above.
This implies together with the Remark~\ref{value_less_than_degree} that $e^*_o(\omega)\le d-2$. Hence the number of single loops
\[\Bigl(o^*-\frac{e^*_o}{2}\Bigr)(\omega)=\deg_{v_o}(\omega) - e^*_o(\omega)\ge \deg_{v_o}(\omega)-d+2,
\]
and we define
\[
o^*(\omega')= \max \{ \deg_{v_o}(\omega)-d+2,0 \}.
\]
Finally we check that the condition (\ref{condition_for_omega_prime}) is fulfilled.
\[
\text{deg}_{v_o}(\omega)-\text{deg}_{v_o}(\omega')\le \text{deg}_{v_o}(\omega)-(\deg_{v_o}(\omega)-d+2)\le d-2.
\]
This completes the proof.
\end{proof}

\section{Examples on small graphs}\label{sect_small_examples}

We conclude the article with some examples of indecomposable elements for special cases of graphs with small first Betti number $g$.

\begin{table}[tb]
{\scriptsize
\begin{tabular}{llll}
\begin{minipage}[t]{0.197\textwidth}
\begin{tabular}[t]{|r|r|r|}
\hline
\multicolumn{3}{|l|}{\cellcolor{pink} $g=1$}\\
\hline
  $d$ & generator & $\#$ \\
\hline
 1 & (0) & 2\\
 2 & (2) & 1\\
\hline
\end{tabular}\\
\vskip 0.1em
\begin{tabular}[t]{|r|r|r|}
\hline
\multicolumn{3}{|l|}{\cellcolor{pink} $g=2$}\\
\hline
  $d$ & generator & $\#$ \\
\hline
$1$ & $(0, 0, 0)$ & $4$ \\
  $2$ & $(0, 2, 2)$ & $1$ \\
  $2$ & $(2, 0, 2)$ & $3$ \\
  $2$ & $(2, 2, 0)$ & $3$ \\
  $3$ & $(2, 2, 2)$ & $4$ \\
\hline
\end{tabular}\\
\vskip 0.1em
\begin{tabular}[t]{|r|r|r|}
\hline
\multicolumn{3}{|l|}{\cellcolor{pink} $g=3$}\\
\hline
  $d$ & generator & $\#$ \\
\hline
  $1$ & $(0, 0, 0, 0, 0)$ & $8$ \\
  $2$ & $(0, 0, 0, 2, 2)$ & $3$ \\
  $2$ & $(0, 2, 2, 0, 2)$ & $3$ \\
  $2$ & $(0, 2, 2, 2, 0)$ & $3$ \\
  $2$ & $(2, 0, 2, 0, 2)$ & $9$ \\
  $2$ & $(2, 0, 2, 2, 0)$ & $9$ \\
  $2$ & $(2, 2, 0, 0, 0)$ & $9$ \\
  $2$ & $(2, 2, 0, 2, 2)$ & $1$ \\
  $3$ & $(0, 2, 2, 2, 2)$ & $8$ \\
  $3$ & $(2, 0, 2, 2, 2)$ & $16$ \\
  $3$ & $(2, 2, 2, 0, 2)$ & $16$ \\
  $3$ & $(2, 2, 2, 2, 0)$ & $16$ \\
  $3$ & $(2, 2, 2, 2, 2)$ & $8$ \\
  $4$ & $(2, 2, 2, 2, 4)$ & $9$ \\
  $4$ & $(2, 2, 2, 4, 2)$ & $9$ \\
  $4$ & $(2, 4, 2, 2, 2)$ & $9$ \\
  $4$ & $(4, 2, 2, 2, 2)$ & $27$ \\
  \hline
\end{tabular}
\end{minipage}
&
\begin{minipage}[t]{0.231\textwidth}
\begin{tabular}[t]{|r|r|r|}
\hline
\multicolumn{3}{|l|}{\cellcolor{pink} $g=4$ and $d \le 2$}\\
\hline
  $d$ & generator & $\#$ \\
\hline
  $1$ & $(0, 0, 0, 0, 0, 0, 0)$ & $16$ \\
  $2$ & $(0, 0, 0, 0, 0, 2, 2)$ & $9$ \\
  $2$ & $(0, 0, 0, 2, 2, 0, 2)$ & $9$ \\
  $2$ & $(0, 0, 0, 2, 2, 2, 0)$ & $9$ \\
  $2$ & $(0, 2, 2, 0, 2, 0, 2)$ & $9$ \\
  $2$ & $(0, 2, 2, 0, 2, 2, 0)$ & $9$ \\
  $2$ & $(0, 2, 2, 2, 0, 0, 0)$ & $9$ \\
  $2$ & $(0, 2, 2, 2, 0, 2, 2)$ & $1$ \\
  $2$ & $(2, 0, 2, 0, 2, 0, 2)$ & $27$ \\
  $2$ & $(2, 0, 2, 0, 2, 2, 0)$ & $27$ \\
  $2$ & $(2, 0, 2, 2, 0, 0, 0)$ & $27$ \\
  $2$ & $(2, 0, 2, 2, 0, 2, 2)$ & $3$ \\
  $2$ & $(2, 2, 0, 0, 0, 0, 0)$ & $27$ \\
  $2$ & $(2, 2, 0, 0, 0, 2, 2)$ & $3$ \\
  $2$ & $(2, 2, 0, 2, 2, 0, 2)$ & $3$ \\
  $2$ & $(2, 2, 0, 2, 2, 2, 0)$ & $3$ \\
\hline
\end{tabular}\\
\vskip 0.1em
\begin{tabular}[t]{|r|r|r|}
\hline
\multicolumn{3}{|l|}{\cellcolor{pink} $g=4$ and $d = 3$}\\
\hline
  $d$ & generator & $\#$ \\
\hline
  $3$ & $(0, 0, 0, 2, 2, 2, 2)$ & $32$ \\
  $3$ & $(0, 2, 2, 0, 2, 2, 2)$ & $32$ \\
  $3$ & $(0, 2, 2, 2, 2, 0, 2)$ & $32$ \\
  $3$ & $(0, 2, 2, 2, 2, 2, 0)$ & $32$ \\
  $3$ & $(0, 2, 2, 2, 2, 2, 2)$ & $16$ \\
  $3$ & $(2, 0, 2, 0, 2, 2, 2)$ & $64$ \\
  $3$ & $(2, 0, 2, 2, 2, 0, 2)$ & $64$ \\
  $3$ & $(2, 0, 2, 2, 2, 2, 0)$ & $64$ \\
  $3$ & $(2, 0, 2, 2, 2, 2, 2)$ & $32$ \\
  $3$ & $(2, 2, 0, 2, 2, 2, 2)$ & $16$ \\
  $3$ & $(2, 2, 2, 0, 2, 0, 2)$ & $64$ \\
  $3$ & $(2, 2, 2, 0, 2, 2, 0)$ & $64$ \\
  $3$ & $(2, 2, 2, 0, 2, 2, 2)$ & $32$ \\
  $3$ & $(2, 2, 2, 2, 0, 0, 0)$ & $64$ \\
  $3$ & $(2, 2, 2, 2, 0, 2, 2)$ & $16$ \\
  $3$ & $(2, 2, 2, 2, 2, 0, 2)$ & $32$ \\
  $3$ & $(2, 2, 2, 2, 2, 2, 0)$ & $32$ \\
  $3$ & $(2, 2, 2, 2, 2, 2, 2)$ & $16$ \\
\hline
\end{tabular}
\end{minipage}
&
\begin{tabular}[t]{|r|r|r|}
\hline
\multicolumn{3}{|l|}{ \cellcolor{pink} $g=4$, $d = 4$}\\
\hline
  $d$ & generator & $\#$ \\
\hline
  $4$ & $(0, 2, 2, 2, 2, 2, 4)$ & $27$ \\
  $4$ &  $(0, 2, 2, 2, 2, 4, 2)$ & $27$ \\
  $4$ &  $(0, 2, 2, 4, 2, 2, 2)$ & $27$ \\
  $4$ &  $(0, 4, 4, 2, 2, 2, 2)$ & $27$ \\
  $4$ &  $(2, 0, 2, 2, 2, 2, 4)$ & $45$ \\
  $4$ &  $(2, 0, 2, 2, 2, 4, 2)$ & $45$ \\
  $4$ &  $(2, 0, 2, 4, 2, 2, 2)$ & $45$ \\
  $4$ &  $(2, 2, 2, 0, 2, 2, 4)$ & $45$ \\
  $4$ &  $(2, 2, 2, 0, 2, 4, 2)$ & $45$ \\
  $4$ &  $(2, 2, 2, 2, 0, 4, 4)$ & $9$ \\
  $4$ &  $(2, 2, 2, 2, 2, 2, 4)$ & $27$ \\
  $4$ &  $(2, 2, 2, 2, 2, 4, 2)$ & $27$ \\
  $4$ &  $(2, 2, 2, 2, 4, 0, 4)$ & $45$ \\
  $4$ &  $(2, 2, 2, 2, 4, 2, 2)$ & $81$ \\
  $4$ &  $(2, 2, 2, 2, 4, 4, 0)$ & $45$ \\
  $4$ &  $(2, 2, 2, 4, 2, 0, 2)$ & $45$ \\
  $4$ &  $(2, 2, 2, 4, 2, 2, 0)$ & $45$ \\
  $4$ &  $(2, 2, 2, 4, 2, 2, 2)$ & $27$ \\
  $4$ &  $(2, 2, 2, 4, 2, 2, 4)$ & $9$ \\
  $4$ &  $(2, 2, 2, 4, 2, 4, 2)$ & $9$ \\
  $4$ &  $(2, 2, 4, 2, 2, 2, 2)$ & $81$ \\
  $4$ &  $(2, 4, 2, 0, 2, 2, 2)$ & $45$ \\
  $4$ &  $(2, 4, 2, 2, 2, 0, 2)$ & $45$ \\
  $4$ &  $(2, 4, 2, 2, 2, 2, 0)$ & $45$ \\
  $4$ &  $(2, 4, 2, 2, 2, 2, 2)$ & $27$ \\
  $4$ &  $(2, 4, 2, 2, 2, 2, 4)$ & $9$ \\
  $4$ &  $(2, 4, 2, 2, 2, 4, 2)$ & $9$ \\
  $4$ &  $(2, 4, 2, 4, 2, 2, 2)$ & $9$ \\
  $4$ &  $(4, 0, 4, 2, 2, 2, 2)$ & $135$ \\
  $4$ &  $(4, 2, 2, 0, 2, 2, 2)$ & $135$ \\
  $4$ &  $(4, 2, 2, 2, 2, 0, 2)$ & $135$ \\
  $4$ &  $(4, 2, 2, 2, 2, 2, 0)$ & $135$ \\
  $4$ &  $(4, 2, 2, 2, 2, 2, 2)$ & $81$ \\
  $4$ &  $(4, 2, 2, 2, 2, 2, 4)$ & $27$ \\
  $4$ &  $(4, 2, 2, 2, 2, 4, 2)$ & $27$ \\
  $4$ &  $(4, 2, 2, 4, 2, 2, 2)$ & $27$ \\
  $4$ &  $(4, 4, 0, 2, 2, 2, 2)$ & $27$ \\
\hline
\end{tabular}
&
\begin{tabular}[t]{|r|r|r|}
\hline
\multicolumn{3}{|l|}{ \cellcolor{pink} $g=4$ and $d = 5$}\\
\hline
  $d$ &  generator & $\#$ \\
\hline
  $5$ &  $(2, 2, 2, 4, 2, 4, 4)$ & $32$ \\
  $5$ &  $(4, 4, 2, 4, 2, 2, 2)$ & $64$ \\
  $5$ &  $(4, 4, 2, 4, 2, 4, 4)$ & $16$ \\
\hline
\end{tabular}
\end{tabular}}
\caption{Generators of phylogenetic semigroup of $g$-caterpilar graph.}
\label{table_list_of_gens}
\end{table}

\begin{figure}[htb]
\begin{tabular}{ccccc}
\begin{minipage}[t]{0.25\textwidth}
\begin{center}
\epsfxsize=70pt\epsfbox{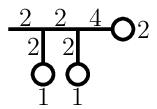}
\end{center}
\caption{An indecomposable element of degree $4$ on the $3$-caterpillar graph.}
\label{fig_indecomposable_of_deg_4_on_3_caterpillar}
\end{minipage}
 &   \phantom{b} &
\begin{minipage}[t]{0.28\textwidth}
\begin{center}
\epsfxsize=100pt\epsfbox{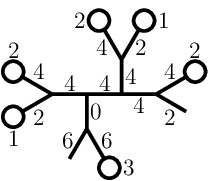}
\end{center}
\caption{An indecomposable element of degree $6$ on a graph with $6$ loops and two leaves.}
\label{fig_indecomposable_of_deg_6_on_genus_6_graph}
\end{minipage}
&   \phantom{b} &
\begin{minipage}[t]{0.30\textwidth}
\begin{center}
\epsfxsize=120pt\epsfbox{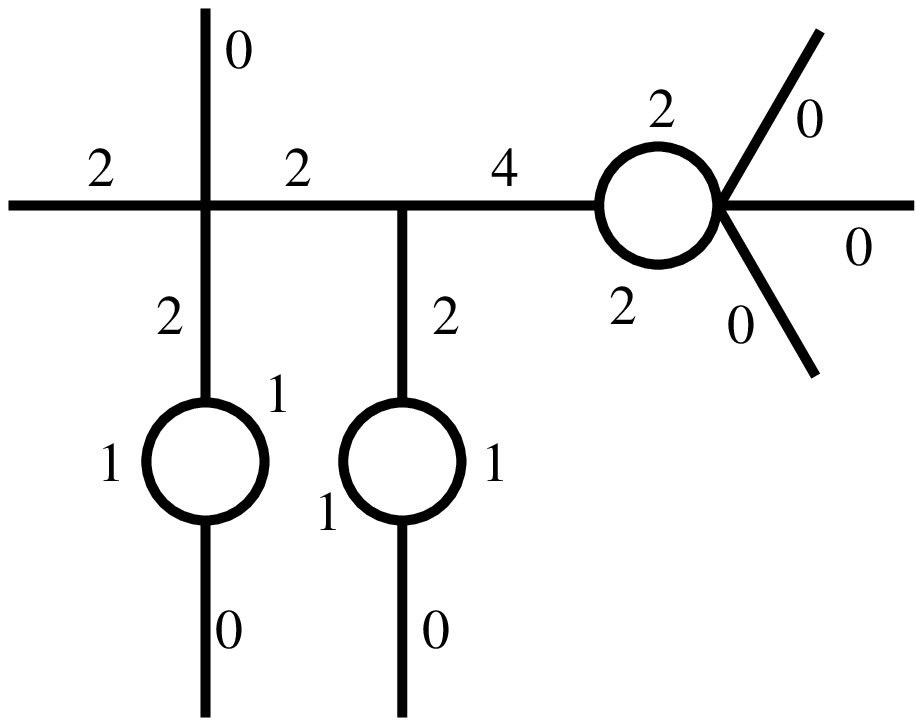}
\end{center}
\caption{An indecomposable element as in Figure~\ref{fig_indecomposable_of_deg_4_on_3_caterpillar} adapted to a graph with no loops and with vertices of high valency.}
\label{fig_indecomposable_with_zeroes}
\end{minipage}
\end{tabular}
\end{figure}

\begin{table}[hbt]
\begin{center}
{\footnotesize
\begin{tabular}{|c|r|r|r|r|r|}
\hline
$d$ & $g=1$ & $g=2$ & $g=3$ & $g=4$ & $g=5$ \\
\hline
 all & $3$ & $15$ & $163$ & $2708$ & $49187$\\
\hline
 $1$ & $2$ & $4$ & $8$ & $16$ & $32$\\
 $2$ & $1$ & $7$ & $37$ & $175$ & $781$\\
 $3$ &     & $4$ & $64$ & $704$ & $6624$\\
 $4$ &     &     & $54$ & $1701$ & $35190$\\
 $5$ &     &     &     & $112$ & $6560$\\
\hline
\end{tabular}}
\end{center}
\caption{Number of generators of the phylogenetic semigroup of $g$-caterpillar graph in each degree.}
\label{table_number_of_gens}
\end{table}

The example on Figure~\ref{fig_indecomposable_of_deg_4_on_3_caterpillar}
   is an indecomposable element of degree $4$ on the $3$-caterpillar graph.
It shows that our bound $d\ge 6$ in Theorem~\ref{thm_even_decompose_on_caterpilar} is necessary,
  and also proves that, in the case $g=3$, the upper bound of Theorem~\ref{thm_upper_bound} is attained.

On Figure~\ref{fig_indecomposable_of_deg_6_on_genus_6_graph}
  there is a degree $6$ indecomposable element on a graph with $6$ loops and one leaf.
This shows that our decomposition Theorem~\ref{thm_even_decompose_on_caterpilar}
  does not work on this non-caterpillar graph.

Despite our examples in Sections~\ref{sect_construction_of_deg_g_plus_1_elts} and \ref{sect_small_examples}
   are indecomposable elements on trivalent graphs that contain loops,
   it is possible to slightly modify those examples to graphs with no loops and to graphs of higher valency.
This is provided by the following two elementary properties, and an example how to apply them is on
   Figure~\ref{fig_indecomposable_with_zeroes}.

\begin{prop}\label{prop_omega_has_0}
   Suppose $e$ is an edge of $\ccG$, and $\omega \in \tau(\ccG)$ is such that $e^*(\omega) = 0$.
   Let $\ccG'$ be the graph obtained from $\ccG$ by removing the edge $e$ and let $\omega' \in \tau(\ccG')$
     be the labelling identical with $\omega $ away from $e$.
   Then $\omega$ is indecomposable in $\tau(\ccG)$ if and only if $\omega'$ is indecomposable in $\tau(\ccG')$.
\end{prop}

\begin{prop}\label{prop_G_has_two_valent}
   Suppose $v$ is a two-valent vertex of $\ccG$ and let $e_1$ and $e_2$ be the two edges containing $v$.
   Then for any $\omega \in \tau(\ccG)$ we have $e_1^*(\omega) = e_2^*(\omega)$.
   Furthermore, $\tau(\ccG)$ is naturally isomorphic to $\tau(\ccG')$,
     where $\ccG'$ is the graph obtained by removing $v$ from $\ccG$ and replacing $e_1$ and $e_2$ with a single edge $e$.
\end{prop}

For $g\le 4$ Table~\ref{table_list_of_gens} lists all generators of $\tau(\ccG)$
 by specifying the possible labellings on all edges except for the loops.
The order of edges goes from left to right, beginning with the leaf, the second is the vertical edge towards the first loop,
   the third is the next horizontal edge, etc.
For instance, the example of Figure~\ref{fig_indecomposable_of_deg_4_on_3_caterpillar} is encoded $(2,2,2,2,4)$
   and can be found in the table for $g=3$
   in the $14^{\text{th}}$ row.
The label on each loop can be set to any integer in the range $\setfromto{\frac{1}{2} c_v}{d-\frac{1}{2} c_v}$.
In the third column $\#$, we specify how many possibilities there are for the labelling on the loops.
An analogous table for $g=5$ would need 359 rows, thus we omit it from this article.

Table~\ref{table_number_of_gens}
  presents the numbers of generators of $\tau(\ccG)$ in each degree, where $\ccG$ is the $g$-caterpillar graph,
   and $g \le 5$.
These calculations were obtained using the convex bodies package in Magma
   \cite{magma}, \cite{magma_handbook_convex_chapter}.
%THINK: post calculations on-line!
\bibliography{g-graphs}

\begin{thebibliography}{Mich11b}

\bibitem[BBK]{magma_handbook_convex_chapter}
Gavin Brown, Jaros\l{}aw Buczy\'nski, and Alexander Kasprzyk.
\newblock Chapter: Convex polytopes and polyhedra.
\newblock In {\em The {M}agma {H}andbook}. University of Sydney.
\newblock Available from http://magma.maths.usyd.edu.au/.

\bibitem[BCP97]{magma}
Wieb Bosma, John Cannon, and Catherine Playoust.
\newblock The {M}agma algebra system. {I}. {T}he user language.
\newblock {\em J. Symbolic Comput.}, 24(3-4):235--265, 1997.
\newblock Computational algebra and number theory (London, 1993). Available for
  use on-line at {\tt http://magma.maths.usyd.edu.au/calc/}.

\bibitem[Bucz12]{buczynska_graphs}
Weronika Buczy\'nska.
\newblock Phylogenetic toric varieties on graphs.
\newblock {\em Journal of Algebraic Combinatorics}, 35(3):421--460, 2012.

\bibitem[BW07]{buczynska_wisniewski}
Weronika Buczy{\'n}ska and Jaros{\l}aw~A. Wi{\'s}niewski.
\newblock On geometry of binary symmetric models of phylogenetic trees.
\newblock {\em J. Eur. Math. Soc. (JEMS)}, 9(3):609--635, 2007.

\bibitem[DBM10]{michalek_donten_bury_Phylogenetic_invariants_for_group-based_m%
odels}
Maria Donten-Bury and Mateusz Michalek.
\newblock Phylogenetic invariants for group-based models.
\newblock arXiv:1011.3236v1, to appear in J. Alg. Stat., 2010.

\bibitem[Falt94]{faltings_proof_for_Verlinde}
Gerd Faltings.
\newblock A proof for the {V}erlinde formula.
\newblock {\em J. Algebraic Geom.}, 3(2):347--374, 1994.

\bibitem[JW92]{jeffrey_weitsman}
Lisa~C. Jeffrey and Jonathan Weitsman.
\newblock Bohr-{S}ommerfeld orbits in the moduli space of flat connections and
  the {V}erlinde dimension formula.
\newblock {\em Comm. Math. Phys.}, 150(3):593--630, 1992.

\bibitem[Kubj10]{kubjas_kaie_hilb_poly_3Kimura_not_the_same}
Kaie Kubjas.
\newblock Hilbert polynomial of the {K}imura 3-parameter model.
\newblock arXiv:1007.3164, 2010.

\bibitem[Mano09]{manon_conformal_blocks}
Christopher~A. Manon.
\newblock The algebra of conformal blocks.
\newblock arXiv:0910.0577v3 [math.AG], 2009.

\bibitem[Mano11]{manon_coordinate_rings}
Christopher~A. Manon.
\newblock Coordinate rings for the moduli of $sl_2(c)$ quasi-parabolic
  principal bundles on a curve and toric fiber products.
\newblock arXiv:1105.2045 [math.AC], 2011.

\bibitem[Mich11a]{michalek_group_based_models_are_pseudo_toric}
Mateusz Micha{\l}ek.
\newblock Geometry of phylogenetic group-based models.
\newblock {\em J. Algebra}, 339:339--356, 2011.

\bibitem[Mich11b]{michalek_toric_geometry_3Kimura}
Mateusz Micha{\l}ek.
\newblock Toric geometry of the 3-{K}imura model for any tree.
\newblock arXiv:1102.4733, to appear in Adv. in Geom., 2011.

\bibitem[Neym71]{Neyman}
Jerzy Neyman.
\newblock Molecular studies in evolution: a source of novel statistical
  problems.
\newblock In Shanti~S. Gupta and James Yackel, editors, {\em Statistical
  Decision Theory and Related Topics}, pages 1--27. Academic Press, New York,
  1971.

\bibitem[PS05]{pachter_sturmfels_alg_stat_for_comp_biology}
Lior Pachter and Bernd Sturmfels.
\newblock Statistics.
\newblock In {\em Algebraic statistics for computational biology}, pages 3--42.
  Cambridge Univ. Press, New York, 2005.

\bibitem[SS05]{sturmfels_sullivant}
Bernd Sturmfels and Seth Sullivant.
\newblock Toric ideals of phylogenetic invariants.
\newblock {\em J. Comput. Biol.}, 12(2):204--228, 2005.

\bibitem[SV10]{sturmfels_velasco_P3_blow_up_and_spinor_vars}
Bernd Sturmfels and Mauricio Velasco.
\newblock Blow-ups of {$\mathbb{P}^{n-3}$} at {$n$} points and spinor
  varieties.
\newblock {\em J. Commut. Algebra}, 2(2):223--244, 2010.

\bibitem[SX10]{sturmfels_xu}
Bernd Sturmfels and Zhiqiang Xu.
\newblock {S}agbi basis and {C}ox-{N}agata rings.
\newblock {\em J. Eur. Math. Soc. (JEMS)}, 12(2):429--459, 2010.

\bibitem[Verl88]{verlinde_fusion_rules}
Erik Verlinde.
\newblock Fusion rules and modular transformations in {$2$}{D} conformal field
  theory.
\newblock {\em Nuclear Phys. B}, 300(3):360--376, 1988.

\end{thebibliography}
%\bibliography{bibliografia}
\bibliographystyle{alpha_four}
\end{document}